\documentclass[11pt]{article}
\usepackage{amsmath,amssymb,amsthm,graphicx,fixmath}
\usepackage{algpseudocode}
\usepackage{algorithm}
\usepackage{setspace}
\usepackage{ifthen}
\usepackage{tikz}
\def\Re{\mathbb R}

\providecommand{\remove}[1]{}

\theoremstyle{plain}
\newtheorem{theorem}{Theorem}[section]
\newtheorem{lemma}[theorem]{Lemma}

\newtheorem{claim}[theorem]{Claim}

\newtheorem{problem}[theorem]{Problem}

\theoremstyle{definition}
\newtheorem{definition}[theorem]{Definition}
\theoremstyle{remark}
\newtheorem{remark}[theorem]{Remark}
\newcommand{\R}{\mathcal{R}}

\newcommand{\E}{\mathcal{E}}

\newcommand{\D}{\mathcal{D}}
\newcommand{\F}{\mathcal{F}}

\begin{document}
	
	\title{Zarankiewicz's Problem via $\epsilon$-t-Nets\footnote{A preliminary version of the paper was published at the SoCG 2024 conference~\cite{KellerS24}.}}
	
%	\iffalse

	\author{%
		Chaya Keller
		\thanks{
			Department of Computer Science,
			Ariel University, Israel.
			Research partially supported by the Israel Science Foundation (grant no.~1065/20) and by the United States -- Israel Binational Science Foundation (NSF-BSF grant no.~2022792).
			\texttt{chayak@ariel.ac.il}
		}
		\and Shakhar Smorodinsky
		\thanks{Department of Computer Science, Ben-Gurion University of the NEGEV, Be'er Sheva 84105, Israel.
				Research partially supported by the Israel Science Foundation (grant no.~1065/20) and by the United States -- Israel Binational Science Foundation (NSF-BSF grant no.~2022792).
			\texttt{shakhar@bgu.ac.il}}
	}

%	\fi
	\date{}
	\maketitle

	\begin{abstract}
		
		The classical Zarankiewicz's problem asks for the maximum number of edges in a bipartite graph on $n$ vertices which does not contain the complete bipartite graph $K_{t,t}$. In one of the cornerstones of extremal graph theory,  K{\H o}v\'{a}ri, S{\'o}s and Tur{\'a}n proved an upper bound of $O(n^{2-\frac{1}{t}})$. In a celebrated result, Fox et al. obtained an improved bound of $O(n^{2-\frac{1}{d}})$ for graphs of VC-dimension $d$ (where $d<t$). Basit, Chernikov, Starchenko, Tao and Tran  improved the bound for the case of semilinear graphs. Recently, Chan and Har-Peled further improved Basit et al.'s bounds and presented (quasi-)linear upper bounds for several classes of geometrically-defined incidence graphs, including a bound of $O(n \log \log n)$ for the incidence graph of points and pseudo-discs in the plane.
		
		In this paper we present a new approach to Zarankiewicz's problem, via $\epsilon$-t-nets -- a recently introduced generalization of the classical notion of $\epsilon$-nets. We show that the existence of `small'-sized $\epsilon$-t-nets implies upper bounds for Zarankiewicz's problem. Using the new approach, we obtain a sharp bound of $O(n)$ for the intersection graph of two families of pseudo-discs, thus both improving and generalizing the result of Chan and Har-Peled from incidence graphs to intersection graphs.
		We also obtain a short proof of the $O(n^{2-\frac{1}{d}})$ bound of Fox et al., and show improved bounds for several other classes of geometric intersection graphs, including a sharp $O(n\frac{\log n}{\log \log n})$ bound for the intersection graph of two families of axis-parallel rectangles. 
	\end{abstract}
	
%	\newpage
	
	\section{Introduction} \label{sec:intro}
	
	\medskip \noindent \textbf{Zarankiewicz's problem.} 
	A central research area in extremal combinatorics is \emph{Tur\'{a}n-type questions}, which ask for the maximum number of edges in a graph on $n$ vertices that does not contain a copy of a fixed graph $H$. This research direction was initiated in 1941 by Tur\'{a}n, who showed that the maximum number of edges in a $K_r$-free graph on $n$ vertices is $(1-\frac{1}{r-1}+o(1))\frac{n^2}{2}$.
	%, obtained by a complete $(r-1)$-partite graph with equal part sizes. 
	Soon after, Erd\H{o}s, Stone and Simonovits solved the problem for all non-bipartite graphs $H$. They showed that the maximum number is $(1-\frac{1}{\chi(H)-1}+o(1))\frac{n^2}{2}$, where $\chi(H)$ is the chromatic number of $H$. 
	
	The bipartite case turned out to be significantly harder, and the question is still widely open for most bipartite graphs (see the survey~\cite{Sudakov10}). The case of $H$ being a complete bipartite graph was first studied by Zarankiewicz in 1951:
	\begin{problem}[Zarankiewicz's problem] What is the maximum number of edges in a $K_{t,t}$-free bipartite graph on $n$ vertices? 
	\end{problem}
	In one of the cornerstone results of extremal graph theory, K{\H o}v\'{a}ri, S{\'o}s and Tur{\'a}n~\cite{KST54} proved an upper bound of $O(n^{2-\frac{1}{t}})$. This bound is sharp for $t=2,3$. For $t=2$, a tightness example is the incidence graph of points and lines in a finite projective plane. 
	%$F_p$ -- i.e., a bipartite graph whose vertex sets are the points and the lines of $F_p$, where a `point vertex' is connected to a `line vertex' if the point is incident to the line. This graph is clearly $K_{2,2}$-free (as two lines cannot be incident to the same two points), its number of vertices is $n=2(p^2+p+1)$, and its number of edges is $\Theta(p^3)=\Theta(n^{2-\frac{1}{t}})$. 
	Similarly, a tightness example for $t=3$ is an incidence graph of points and spheres (of a carefully chosen fixed radius) in a three-dimensional finite affine space (see~\cite{Brown66}). The question whether the K{\H o}v\'{a}ri-S{\'o}s-Tur{\'a}n theorem is tight for $t \geq 4$ is one of the central open problems in extremal graph theory.
	
	\medskip \noindent \textbf{VC-dimension and Fox et al.'s bound for Zarankiewicz's problem.} The VC-dimension of a hypergraph is a measure of its complexity, which plays a central role in statistical learning, computational geometry, and other areas of combinatorics and computer science (see, e.g.,~\cite{AHW87,BEHW89,MV18}). The VC-dimension of a hypergraph $H=(V,\E)$ is the largest integer $d$ for which there exists $S \subset V, |S|=d$, such that for every subset $B \subset S$, one can find a hyperedge $e \in \E $ with $e \cap S=B$. Such a set $S$ is said to be \emph{shattered}. The \emph{primal shatter function} of $H$ is $\pi_H(m)= \max_{S \subset V, |S|=m}| \{S \cap e: e \in \E\}   |$, and by the Perles-Sauer-Shelah lemma (see~\cite{MATOUSEK}), if $H$ has VC-dimension $d$ then $\pi_H(m)\leq \Sigma_{i=0}^d \binom{m}{i}=O(m^d)$. The dual hypergraph of a hypergraph $H=(V,\E)$ is $H^*=(V^*,\E^*)$, where $V^*=\E$ and each $v \in V$ induces the hyperedge $e_v \in \E^*$, where $e_v=\{  e \in \E : v \in e \}$. When the VC-dimension of $H$ is $d$, the VC-dimension of $H^* $ is denoted by $d^*$. 
	
	Any bipartite graph $G=G_{A,B}$ with vertex set $V(G)=A \cup B$ and edge set $E(G) \subset A \times B$, defines two hypergraphs: the primal hypergraph $H_G=(A,\E_B)$, where $\E_B=\{  N(b):b \in B  \}$ is the collection of the open neighborhoods of the vertices in $B$, and the dual hypergraph $H_G^*=(B,\E_A)$, defined similarly. The VC-dimension of $G$ is defined as the VC-dimension of $H_G$, and the dual VC-dimension of $G$ is defined as the VC-dimension of $H_G^*$. The shatter function $\pi_G$ and the dual shatter function $\pi_G^*$ of $G$, are the shatter functions of $H_G$ and of $H_G^*$, respectively. 
	
	In a remarkable result, Fox, Pach, Sheffer, Suk and Zahl~\cite{FPSSZ17} improved the bound of the K{\H o}v\'{a}ri-S{\'o}s-Tur{\'a}n theorem for graphs with VC-dimension at most $d$ (for $d<t$). They showed:
	\begin{theorem}[\cite{FPSSZ17}]
		\label{thm:main-boundedvc}
		Let $t \geq 2$ and let $G_{A,B}$ be a bipartite graph with $|A|=m$ and $|B|=n$, satisfying $\pi_G(\ell)=O(\ell^d)$ and $\pi_G^*(\ell)=O(\ell^{d^*})$ for all $\ell$. If $G$ is $K_{t,t}$-free, then $$ |E(G)|=O_{t,d,d^*}(\min \{  mn^{1-\frac{1}{d}}+n,  nm^{1-\frac{1}{d^*}}+m  \}) .$$
	\end{theorem}
	Theorem~\ref{thm:main-boundedvc} spawned several follow-up papers. Janzer and Pohoata~\cite{JP20} obtained an improved bound of $o(n^{2-\frac{1}{d}})$ for graphs with VC-dimension $d$, where $m = n$ and $t \geq d>2$, using the hypergraph removal lemma~\cite{Gowers07}. Do~\cite{Do19} and Frankl and Kupavskii~\cite{FK21} obtained improved bounds when $t$ tends to infinity with $n$.
	
	\medskip \noindent \textbf{Improved bounds for Zarankiewicz's problem for incidence graphs.} An \emph{incidence graph} is a bipartite graph whose vertex set is a union of a set of points and a set of geometric objects, where the edges connect points to objects to which they are incident. Problems on incidence graphs are central in combinatorial and computational geometry. For example, the classical Erd\H{o}s' unit distances problem asks for an upper bound on the number of edges in an incidence graph of points and unit circles. %Furthermore, they are closely related to algorithmic problems in computational geometry, such as range searching and Hopcroft's problem (see~\cite{AE99,CZ22}).
	
	Incidence graphs are naturally related to Zarankiewicz's problem. Indeed, the incidence graph of points and lines is $K_{2,2}$-free, and thus, the K{\H o}v\'{a}ri-S{\'o}s-Tur{\'a}n theorem implies that its number of edges is $O(n^{3/2})$. While this bound is tight for the incidence graph of the finite projective plane, the classical Szemer\'{e}di-Trotter theorem (see~\cite{MATOUSEK}) asserts that a stronger bound of $O(n^{4/3})$ holds in the real plane. 
	
	Motivated by this relation, Basit, Chernikov, Starchenko, Tao, and Tran~\cite{BCS+21} studied incidence graphs of points and axis-parallel boxes in $\mathbb{R}^d$, under the additional assumption that they are $K_{t,t}$-free. They obtained an $O(n \log^{2d} n)$ bound in $\mathbb{R}^d$, and a sharp $O(n\frac{\log n}{\log \log n})$ bound for dyadic axis-parallel rectangles in the plane. Related incomparable results were independently obtained by Tomon and Zakharov~\cite{TZ21}; see below.
	
	Recently, Chan and Har-Peled~\cite{CH23} initiated a systematic study of Zarankiewicz's problem for incidence graphs of points and~various geometric objects. They obtained an $O(n(\frac{\log n}{\log \log n})^{d-1})$ bound for the incidence graph of points and axis-parallel boxes in $\mathbb{R}^d$ and observed that a matching lower bound construction appears in a classical paper of Chazelle (\cite{Chazelle90}; see also \cite{Tomon23}). They also obtained an $O(n \log \log n)$ bound for points and pseudo-discs in the plane, and bounds for points and halfspaces, balls, shapes with `low union complexity', and more. The proofs in~\cite{CH23} use a variety of techniques, including shallow cuttings, a geometric divide-and-conquer, and biclique covers.
	
	\medskip \noindent \textbf{$\epsilon$-nets and $\epsilon$-$t$-nets.} Given a hypergraph $H=(V,\E)$ and $\epsilon>0$, an \emph{$\epsilon$-net} for $H$ is a set $S \subset V$ such that any hyperedge $e \in \E$ of size $\geq \epsilon |V|$ contains a vertex from $S$. The notion of $\epsilon$-nets was introduced by Haussler and Welzl~\cite{hw-ensrq-87} who proved that any finite hypergraph with VC-dimension $d$ admits an $\epsilon$-net of size $O((d/\epsilon)\log(d/\epsilon))$ (a bound that was later improved to $O((d/\epsilon)\log(1/\epsilon))$ in~\cite{KPW92}). $\epsilon$-nets were studied extensively and have found applications in diverse areas of computer science,
	including machine learning, algorithms, computational geometry, and social choice (see, e.g.,~\cite{ABKKW06,AFM18,BEHW89,Chan18}).
	
	In 2020, Alon et al.~\cite{AJKSY22} introduced the following notion of \emph{$\epsilon$-$t$-nets}, generalizing $\epsilon$-nets and the notion of \emph{$\epsilon$-Mnets} that was studied by Mustafa and Ray~\cite{MustafaR17} and by Dutta et al.~\cite{DuttaGJM19}:
	\begin{definition}
		Let $\epsilon \in (0,1)$ and $t \in \mathbb{N} \setminus \{0\}$ be fixed parameters, and let 
		$H=(V,\E)$ be a hypergraph on $n$ vertices. A set $N \subset \binom{V}{t}$ of $t$-tuples of vertices is called an \emph{$\epsilon$-$t$-net} if any hyperedge $e \in \E$ with $|e| \geq \epsilon n$ contains at least one of the t-tuples in $N$.
	\end{definition}
	%-- sets $S$ of $t$-tuples of vertices of $H$ such that any hyperedge $e \in \E$ of size $\geq \epsilon |V|$ fully contains a $t$-tuple from $S$. 
	Alon et al. \cite{AJKSY22} proved the following:
	\begin{theorem}\label{thm:eps-t-net}
		For every $\epsilon \in (0,1)$, $C>0$, and $t,d,d^* \in \mathbb{N} \setminus \{0\}$, there exists $C_1=C_1(C,d^*)$ such that the following holds. Let $H$ be a hypergraph on at least $C_1((t-1)/\epsilon)^{d^*}$vertices with VC-dimension $d$ and dual shatter function $\pi^*_H(m) \leq C \cdot m^{d^*}$. Then $H$ admits an $\epsilon$-$t$-net of size $O((d(1+\log t)/\epsilon)\log(1/\epsilon))$, all of which elements are pairwise disjoint.	 
	\end{theorem}
	In addition, the paper \cite{AJKSY22} studied the existence of small-sized $\epsilon$-$t$-nets in various geometric settings in which it is known that the classical $\epsilon$-net has size $O(1/\epsilon)$. In particular, they showed that the intersection hypergraph of two families of pseudo-discs and the dual incidence hypergraph of points and a family of
	regions with a linear union complexity, admit $\epsilon$-$2$-nets of size $O(1/\epsilon)$, provided that they have at least $2/\epsilon$ vertices.
	
	\subsection{Our results}

	\medskip \noindent \textbf{Zarankiewicz's problem via $\epsilon$-$t$-nets.} The basic observation underlying our results is a surprising connection between $\epsilon$-$t$-nets and Zarankiewicz's problem. Consider a bipartite graph $G=G_{A,B}=(A \cup B, E)$ where $|A|=m,|B|=n$, and assume that for some $\epsilon$, the primal hypergraph $H_G=(A,\E_B)$ (recall that $\E_B=\{  N(b):b \in B  \}$ is the collection of the open neighborhoods of the vertices in $B$) admits an $\epsilon$-$t$-net $S$ of size $s$. We can partition the vertex set $B$ into the set $B'$ of `heavy' vertices that have degree at least $\epsilon m$ in $G$, and the set $B''$ of `light' vertices that have degree less than $\epsilon m$.  
	
	We observe that the neighborhood $N(b)$ of any `heavy' vertex $b$ must contain a $t$-tuple from $S$. Since $G$ is $K_{t,t}$-free, any $t$-tuple from $S$ is contained in at most $t-1$ neighborhoods $N(b_i)$. Hence, the number of `heavy' vertices is at most $s(t-1)$. This immediately yields the bound
	\begin{equation*}\label{Eq:First-step-intro}
		|E(G)| \leq n \lfloor \epsilon m \rfloor + s(t-1)m, 
	\end{equation*}
	where the first term is the contribution of the `light' vertices and the second term is the contribution of the `heavy' vertices.
	
	%If the dual hypergraph $H_G^*=H(B,\E_A)$ admits a small-sized $\epsilon'$-$t$-net $S'$ of size $s'$ as well (possibly with different $\epsilon',s'$), then we can repeat the procedure described above and divide the vertex set $A$ into the set $A'$ of `heavy' vertices that have at least $\epsilon' n$ neighbors and the set $A''$ of `light' vertices that have less than $\epsilon' n$ neighbors. By the same argument as above, $|A'|\leq s'(t-1)$. Hence, we can obtain the bound  
	%\[
	%|E(G)| \leq n(\epsilon m) + m(\epsilon' n) + ss'(t-1)^2, 
	%\]
	%where the first term bounds the contribution of the `light' vertices from $B$, the second term bounds the contribution of the `light' vertices from $A$, and the third term bounds the contribution of the `heavy'-`heavy' edges.
	
	%Finally, we can further reduce the third term by observing that the set of `heavy'-`heavy' edges is the edge set of the induced subgraph of $G$ on the vertex set $A' \cup B'$, which is $K_{t,t}$-free, and thus, we may be able to apply to it the above partioning once again, provided that the corresponding primal and dual hypergraphs admit small-sized $\epsilon$-$t$-nets.
	
	Building upon that and several more observations, we develop a \emph{recursive approach} which allows obtaining bounds for Zarankiewicz's problem using results on $\epsilon$-$t$-nets (see Theorem~\ref{thm:recursive} below).
	
	\medskip \noindent \textbf{A short proof of Fox et al.'s bound.} Our first application of the $\epsilon$-$t$-net approach is a short proof of the bound of Fox, Pach, Sheffer, Suk and Zahl~\cite{FPSSZ17} (Theorem~\ref{thm:main-boundedvc} above). By combining the strategy described above with Theorem~\ref{thm:eps-t-net}, we prove:
	\begin{theorem}\label{thm:VCour}
		Let $t \geq 2$ and let $G_{A,B}$ be a bipartite graph with $|A|=m$ and $|B|=n$, satisfying $\pi_G(\ell)=O(\ell^d)$ and $\pi_G^*(\ell)=O(\ell^{d^*})$ for all $\ell$. If $G$ is $K_{t,t}$-free, then we have $$ |E(G)|=O_{t,d,d^*}(\min \{  mn^{1-\frac{1}{d}}+n^{1+\frac{1}{d}} \log n,  nm^{1-\frac{1}{d^*}}+m^{1+\frac{1}{d^*}} \log m  \}) .$$	
	\end{theorem}
	%Consider a bipartite graph $G=(A \cup B,E)$ whose primal hypergraph $H_G=(A,\E_B)$ has VC-dimension $d$. By Theorem~\ref{thm:eps-t-net}, $H_G$ admits an $\epsilon$-$t$-net of size $O((d(1+\log t)/\epsilon)\log(1/\epsilon))$, for any $\epsilon \geq Cm^{-1/d^*}$, where $C=C(t,d,d^*)$.  Hence,~\eqref{Eq:First-step-intro} yields the bound
	%\[
	%|E(G)| \leq O(nm^{1-\frac{1}{d^*}} + dt \log t \cdot m^{1+\frac{1}{d^*}}\log m),
	%\] 
	The bound of Theorem~\ref{thm:VCour} matches the bound of Theorem~\ref{thm:main-boundedvc}, whenever $d,d^*>2$ and $n,m$ do not differ `too much'.
	Interestingly, when $n=m$ and $d^*=d$, our proof strategy can be combined with the proof strategy of Fox et al.~\cite{FPSSZ17} to obtain a slightly better bound (see Appendix~\ref{App:Fox}). 
	
	\medskip \noindent \textbf{Zarankiewicz's problem for intersection graphs.} The \emph{intersection graph} of a family $\mathcal{F}$ of geometric objects is a graph whose vertex set is $\mathcal{F}$, and whose edges connect pairs of objects whose intersection is non-empty.  In the general (i.e., non-bipartite) setting, $K_t$-free intersection graphs of geometric objects were studied extensively, and have applications to the study of quasi-planar topological graphs (see, e.g.,~\cite{FoxP12}). 
	
	Generalizing the systematic study of Zarankiewicz's problem for incidence graphs initiated by Chan and Har-Peled~\cite{CH23}, we study Zarankiewicz's problem for \emph{bipartite} intersection graphs of geometric objects -- i.e., the maximum number of edges in a $K_{t,t}$-free graph $G_{A,B}=(A \cup B,E)$, where $A,B$ are families of geometric objects, and objects $x \in A, y\in B$ are connected by an edge if their intersection is non-empty. Obviously, incidence graphs are the special case where $A$ consists of a set of points. 
	
	It is important to note that this setting (i.e., bipartite intersection graphs) is different from the (standard) intersection graph of the family $A \cup B$, in which intersections inside $A$ and inside $B$ are also taken into account. The stark difference is exemplified well in the case of families $A,B$ of segments in the plane. If the bipartite intersection graph of $A,B$ is $K_{2,2}$-free, the results of Fox et al.~\cite{FPSSZ17} imply the upper bound $O(n^{4/3})$ on its number of edges, and this bound is tight, in view of the tightness examples of the Szemer\'{e}di-Trotter theorem. On the other hand, if the (standard) intersection graph of $A \cup B$ is $K_{2,2}$-free, then the results of Fox and Pach~\cite{FP08} imply an upper bound of $O(n)$ on its number of edges (see also~\cite{MustafaP16}). In fact, an $O(n)$ upper bound was obtained by Fox and Pach~\cite{FP14} even in the much more general case of string graphs. Our setting is the natural generalization of incidence graphs, in which only point-object incidences are taken into account, but not intersections between the objects.

	\medskip \noindent \textbf{A sharp bound for Zarankiewicz's problem for intersection graphs of pseudo-discs.} A \emph{family of pseudo-discs} is a family of simple closed Jordan regions in the plane such that the boundaries of any two regions intersect in at most two points. For example, a family of homothets (scaled translation copies) of a given convex body in the plane is a family of pseudo-discs. As a second application of the strategy described above, we obtain a linear upper bound for Zarankiewicz's problem for the intersection graph of two families of pseudo-discs.
	
	\begin{theorem}\label{thm:pd-intro}
		Let $t \geq 2$ and let $G=G_{A,B}$ be the bipartite intersection graph of families $A,B$ of pseudo-discs, with $|A|=|B|=n$. If $G$ is $K_{t,t}$-free then $|E(G)|=O(t^6n)$.
	\end{theorem}
	In fact, we show that the assertion of Theorem~\ref{thm:pd-intro} holds (with a slightly weaker bound of $O(t^8n)$) for a wider class of bipartite intersection graphs of any two families of so-called \emph{non-piercing regions} -- namely, families $\mathcal{F}$ of regions in the plane such that for any $S,T \in \mathcal{F}$, the region $S \setminus T$ is connected.
	
	Theorem~\ref{thm:pd-intro} improves and generalizes a result of Chan and Har-Peled~\cite{CH23} who obtained an upper bound of $O(n\log \log n)$ for the incidence graph of points and~pseudo-discs. 
	%In the special setting of~\cite{CH23}, of points and pseudo-discs, we present also a proof of the stronger bound $O(t^4 n)$.   
	
	In order to prove Theorem \ref{thm:pd-intro} we show that the primal and the dual hypergraphs of $G$ admit $\epsilon$-$t$-nets of size $O_t(1/\epsilon)$ for all $ n \geq \frac{2t}{\epsilon}$ (see Theorem~\ref{thm:pseudo_discs}). Thus, we also extend the results of~\cite{AJKSY22}, and believe that this might be of independent interest.
	
	Theorem~\ref{thm:pd-intro} demonstrates the added value of the new $\epsilon$-$t$-net approach over previous approaches that used shallow cuttings. There are settings, like intersection graphs of pseudo-discs, for which one can show the existence of a linear-sized $\epsilon$-$t$-net, while the existence of shallow cuttings is not known. In such settings, the $\epsilon$-$t$-net approach yields stronger bounds than previous techniques.
	
	An interesting problem which is left open is whether the dependence on $t$ in Theorem~\ref{thm:pd-intro} can be improved. It seems that the right dependence should be linear, like in the bounds of Chan and Har-Peled~\cite{CH23}.  
	
	%To prove the theorem, we first extend the results of~\cite{AJKSY22} to show that the primal and the dual hypergraphs of $G$ admit $\epsilon$-$t$-nets of size $O_t(1/\epsilon)$ for all $\epsilon \geq tm$ (resp., all $\epsilon \geq tn$). As the properties of being a family of pseudo-discs and of being $K_{t,t}$-free are hereditary, this allows us to apply all three steps of the strategy described above, with appropriately chosen values of $\epsilon,\epsilon'$, to obtain a recursive formula of the form
	%\[
	%|f(m,n)| \leq O_t(1) (m+n) + f(\frac{m}{2},\frac{n}{2}),
	%\]
	%where $f(m,n)$ is the maximum number of edges in a graph that satisfies the hypothesis of the theorem. The recursive formula directly yields the assertion of the theorem. 

	\medskip \noindent \textbf{A sharp bound for intersection graphs of axis-parallel rectangles.} The main class of incidence graphs studied in the previous papers~\cite{BCS+21,CH23,TZ21} is incidence graphs of points and~axis-parallel rectangles (and more generally, axis-parallel boxes in $\mathbb{R}^d$). 
	
	Tomon and Zakharov~\cite{TZ21} studied the related notion of $K_{t,t}$-free intersection graphs of two families of axis-parallel boxes, under the stronger assumption that a $K_{t,t}$ including intersections inside $A$ and $B$ is also forbidden. They obtained a bound of $O(n \log^{2d+3} n)$ for $K_{t,t}$-free intersection graphs of axis-parallel boxes in $\mathbb{R}^d$, as well as a bound of $O(n \log n)$ for $K_{2,2}$-free incidence graphs of points and axis-parallel rectangles in the plane.
	
	We obtain the following essentially tight bound:
	\begin{theorem}\label{thm:rect-intro}
		Let $t \geq 2$, and let $G=G_{A,B}$ be the bipartite intersection graph of families $A,B$ of axis-parallel rectangles in general position\footnote{Here general position means that no two edges of rectangles in $A \cup B$ lie on the same vertical or horizontal line.}, with $|A|=|B|=n$. If $G$ is $K_{t,t}$-free, then $|E(G)|=O\left(t^6 n \frac{\log n}{\log \log n}\right)$.
	\end{theorem} 
	As follows from a lower bound given in~\cite{BCS+21}, this result is sharp in terms of the dependence on $n$ even in the special case where one of the families consists of points and the other consists of dyadic axis-parallel rectangles.
	
	An interesting open question is whether the dependence on $t$ can be improved. It seems that the `right' dependence should be linear.
	
	\medskip
	
	In the case of bipartite intersections graphs of families of axis-parallel rectangles, and even in the more basic case of incidence graphs of points and axis-parallel rectangles, the currently known bounds on the size of $\epsilon$-$t$-nets do not allow obtaining efficient bounds for Zarankiewicz's problem using our $\epsilon$-$t$-net based strategy. Indeed, among these settings, `small'-sized $\epsilon$-$t$-nets for all $\epsilon \geq c/n$ are known to exist only for the incidence hypergraph of points and axis-parallel rectangles, and the size of the $\epsilon$-$t$-net is $O(\frac{1}{\epsilon}\log{\frac{1}{\epsilon}} \log \log{\frac{1}{\epsilon}})$ (see~\cite[Theorem~6.10]{AJKSY22}). Applying the strategy described above with an $\epsilon$-$t$-net of such size would lead to an upper bound on the number of edges in a $K_{t,t}$-free incidence graph of points and axis-parallel rectangles, that is no better than $O(n\log n \log \log n)$.
	
	In order to obtain the stronger (and tight) bound of $O(n\frac{\log n}{\log \log n})$ in the more general setting of bipartite intersection graphs of two families of axis-parallel rectangles, we combine the result of Chan and Har-Peled~\cite{CH23} with a combinatorial argument. Our technique allows us also to obtain an $O(t^6 n)$ upper bound on the number of edges in the bipartite intersection graph of two families of $n$ axis-parallel \emph{frames} (i.e., boundaries of rectangles) in the plane, and an improved bound of $O(t^4n)$ on the number of edges in the intersection graph of points and pseudo-discs (for which Chan and Har-Peled~\cite{CH23} obtained the bound $O(n\log \log n)$). 
	
	\subsection{Follow-up work}
	
	After the initial version of this paper was posted on arXiv, Hunter et al.~\cite{HMST24} observed that a weaker version of Theorem~\ref{thm:pd-intro} -- the upper bound for $K_{t,t}$-free bipartite intersection graphs of two families of pseudo-discs, can be derived from a recent work of Gir\~{a}o and Hunter~\cite{GH23} on \emph{degree-boundedness} of $K_{t,t}$-free graphs. The upper bound which can be derived using this method is $O(t^{12500}n)$, compared to $O(t^6 n)$ in Theorem~\ref{thm:pd-intro}. 
	
Specifically, let $G$ be the bipartite intersection graph of two families $A,B$ of $n$ pseudo-discs, and assume that $|E(G)|>O(t^{12500}n)$, and consequently, the average degree in $G$ is larger than $O(t^{12500})$. A special case of~\cite[Theorem~1.2]{GH23} states that if a $K_{t,t}$-free graph has average degree larger than $O(t^{12500})$, then the graph must contain an induced \emph{balanced} subdivision of $K_5$ -- namely, an induced subdivision of $K_5$ in which every edge of the $K_5$ is replaced by a path of the same length. Thus, $G$ contains such an induced subdivision. As $G$ is bipartite, the balancedness implies by a parity argument that all five vertices of the $K_5$ belong to the same family, say $A$. This implies that the Delaunay graph of the intersection hypergraph whose vertex set is $A$ and whose hyperedges are determined by the elements of $B$, contains a subdivision of $K_5$, and thus, is not planar. However, this contradicts a result of Keszegh~\cite[Theorem~16]{Kes17+}, which states that such a Delaunay graph must be planar for any two families $A,B$ of pseudo-discs.\footnote{We note that a similar argument implies that in the general (i.e., non-bipartite) case, a $K_{t,t}$-free string graph has $O_t(n)$ edges, since it does not contain an induced proper subdivision of $K_5$. This gives an alternative proof of the aforementioned bound of Fox and Pach~\cite{FP14}, alas with a worse dependence on $t$.} 
	
It is worth to mention that upper bounds on the average degree of $K_{t,t}$-free graphs which do not contain an induced 
\emph{proper} subdivision of $K_h$ were obtained in several previous works, starting with the fundamental work of K\"{u}hn and Osthus~\cite{KuhnO04}. However, as was described above, the application of such a result in our case requires a stronger assertion in which proper subdivisions are restricted to  \emph{balanced} subdivisions. Hence, the stronger recent result of Gir\~{a}o and Hunter~\cite{GH23} is needed in order to use this strategy here.

	\subsection{Organization of the paper} 
	
	In Section~\ref{sec:main} we present our new approach to Zarankiewicz's problem via $\epsilon$-$t$-nets and prove Theorems~\ref{thm:VCour} and~\ref{thm:pd-intro}. In Section~\ref{sec:rectangles} we obtain a sharp bound on the number of edges in a $K_{t,t}$-free bipartite intersection graph of two families of axis-parallel rectangles, and use a similar strategy to obtain an improved bound on the number of edges in a $K_{t,t}$-free incidence graph of points and pseudo-discs. Finally, in Appendix~\ref{App:Fox} we compare our Theorem~\ref{thm:VCour} with Theorem~\ref{thm:main-boundedvc} of Fox et al.~\cite{FPSSZ17}.
	
	\section{From $\epsilon$-$t$-Nets to $K_{t,t}$-free Bipartite Graphs}
	\label{sec:main}
	
	\subsection{A recursive upper bound on the size of $K_{t,t}$-free bipartite graphs}
	
	Let $G_{A,B}$ be a bipartite graph, where $|A|=m,|B|=n$. As was shown in the introduction, if for some $\epsilon>0$, the primal hypergraph $H_G=(A,\{N(b)  \}_{b \in B})$ admits an $\epsilon$-$t$-net of size $s$, we may partition the vertex set $B$ into the set $B'$ of `heavy' vertices that have degree at least $\epsilon m$ and the set $B''$ of `light' vertices that have degree less than $\epsilon m$, and observe that since $G_{A,B}$ is $K_{t,t}$-free, $|B'|\leq s(t-1)$. This yields the bound   
	\[
	|E(G)| \leq n \lfloor\epsilon m \rfloor + s(t-1)m, 
	\]
	where the first term is the contribution of the `light' vertices and the second term is the contribution of the `heavy' vertices.
	
	If, in addition, the dual hypergraph $H_G^*=H( B, \{  N(a)\}_{a \in A})$ admits an $\epsilon'$-$t$-net of size $s'$, then we can repeat the procedure described above and partition the vertex set $A$ into the set $A'$ of `heavy' vertices that have degree at least $\epsilon' n$ and the set $A''$ of `light' vertices that have degree less than $\epsilon' n$. By the same argument as above, $|A'|\leq s'(t-1)$. Hence, we can obtain the bound  
	\[
	|E(G)| \leq n \lfloor\epsilon m \rfloor + m \lfloor \epsilon' n \rfloor + ss'(t-1)^2, 
	\]
	where the first term bounds the contribution of the `light' vertices from $B$, the second term bounds the contribution of the `light' vertices from $A$, and the third term bounds the contribution of the `heavy'-`heavy' edges.
	
	Furthermore, we can reduce the third term by observing that the set of `heavy'-`heavy' edges is the edge set of the bipartite graph $G_{A',B'}$ (i.e., the induced subgraph of $G$ on the vertex set $A' \cup B'$), which is $K_{t,t}$-free, and thus, we may be able to apply to it the above partioning once again, provided that the corresponding primal and dual hypergraphs admit `small'-sized $\epsilon$-$t$-nets. This gives rise to the recursive Algorithm~\ref{alg} depicted below.
	
	%In the following Section we levearge the relation obtained in Section \ref{subsec:relation} between $\epsilon$-$t$-nets to $K_{t,t}$-free bipartite fraphs, to obtain a recursive algorithm that `translates' $\epsilon$-$t$-net theorems, to upper bounds on $K_{t,t}$-free bipartite graphs. Then in Section \ref{subsec:pseudo-discs} we demonstrate the algorithm and analyze it in the case of an intersection graph of pseudo-discs.
	
	%\subsection{The Algorithm} \label{subsec:alg}
	%The relation established in Section \ref{subsec:relation} between $\epsilon$-$t$-nets and $K_{t,t}$-free bipartite fraphs, can be leveraged by the following algorithm. The input is a bipartite graph $G_{A,B}$ with $|V(A)|=m$ and $|V(B)|=n$, and the output is an upper bound on $|E(G_{A,B})|$. 
	
	%Some flexible part of the algorithm is the choice of $\epsilon_1, \epsilon_2$ at Step 1, that depends on the exact values of the corresponding $\epsilon$-$t$-net theorem. In Section \ref{subsec:pseudo-discs} we simply choose both of them to be the smallest possible, such that the required $\epsilon$-$t$-net theorem still holds.
	
	%Recall that $G_{A,B}$ defines two hypergraphs:
	%$$H_1=(  A, \{  e_b=\{a \in A : (a,b) \in E(G_{A,B})  \}_{b \in B}   )$$
	%$$H_2=(  B, \{  e_a=\{b \in B : (a,b) \in E(G_{A,B})  \}_{a \in A}   ).$$
	%Both these hypergraphs will be used in the algorithm. Furthermore, for each $A' \subset A, B' \subset B$ denote by $G_{A',B'}$ the subgraph of $G_{A,B}$ induced by $A' \cup B'$.
	
	\begin{algorithm}[H]
		\caption{NumEdges}
		\hspace*{\algorithmicindent} \textbf{Input:} 
		%$m,n,t \in N,G_{A,B}$, with $|A| \leq m, |B| \leq n$\\ 
		$G_{A,B}$, $t$ \\
		\hspace*{\algorithmicindent} \textbf{Output:} Upper bound on $|E(G_{A,B})|$, assuming $G_{A,B}$ is $K_{t,t}$-free
		\label{alg}
		\begin{algorithmic}[1] 
			\State Choose $\epsilon,\epsilon'$
			\State Define $s$ to be the minimum size of an $\epsilon$-$t$-net for $H_G$
			\State Define $s'$ to be the minimum size of an $\epsilon'$-$t$-net for $H^*_G$
			\State Let $A'= \{v \in A:\mathrm{deg}_{G_{A,B}}(v) \geq \epsilon' n\}$  
			\State Let $B'= \{w \in B:\mathrm{deg}_{G_{A,B}}(w) \geq \epsilon m\}$
			\State Return 	$n \lfloor\epsilon m \rfloor + m \lfloor \epsilon' n \rfloor$ +NumEdges$(G_{A',B'})$  		
		\end{algorithmic}
	\end{algorithm}
	
	Note that the choice of $\epsilon, \epsilon'$ at Step 1 of the algorithm is not specified. The optimal choice is determined by the dependence of the size of the smallest $\epsilon$-$t$-net of $H_G$ and of $H^*_G$ on $\epsilon$. In the applications presented below, we simply choose both $\epsilon$ and $\epsilon'$ to be the smallest possible value for which the existence of a `small'-sized $\epsilon$-$t$-net for the corresponding hypergraph is known. 
	
	\medskip \noindent \textbf{Correctness of Algorithm \ref{alg}.}
	
	Let us call a vertex of $A' \cup B'$ \emph{heavy}, and call the other vertices \emph{light}. Algorithm \ref{alg} counts separately the edges of $G_{A,B}$ that involve a light vertex, and the edges that connect two heavy vertices. All the latter edges are counted by the recursion at Step 6. 
	%Note that $|A'| \leq k_2(t-1)$ since any $t$-tuple in the $\epsilon_2$-$t$-net for $H_2$ participates in at most $t-1$ hyperedges of $H_2$, due to being $G_{A,B}$ a $K_{t,t}$-free graph. Similarly, $|B'| \leq k_1(t-1)$.
	
	Regarding the edges of $G_{A,B}$ that involve a light vertex, there are at most $m$ light vertices in $A$, and each of them is involved in at most $\lfloor \epsilon' n \rfloor$ edges of $G_{A,B}$. Similarly, there are at most $n$ light vertices in $B$, and each of them is involved in at most $\lfloor \epsilon m \rfloor$ edges of $G_{A,B}$. This explains the additive term $n \lfloor\epsilon m \rfloor + m \lfloor \epsilon' n \rfloor$ at Step 6. 
	
	\medskip \noindent \textbf{Upper bound for Zarankiewicz's problem for hereditary classes of objects.}
	
	Algorithm~\ref{alg} allows establishing a recursive formula that yields an upper bound for Zarankiewicz's problem for a wide class of graphs. To present the formula in its full generality, a few more definitions and notations are needed.
	
	For a bipartite graph $G=G_{A,B}$ where $|A|=m, |B|=n$, we denote by $f_G(m,k)$ the minimum size of a $\frac{k}{m}$-$t$-net of the primal hypergraph $H$ that corresponds to $G$, and by $f^*_G(n,\ell)$ the minimum size of an $\frac{\ell}{n}$-$t$-net of the dual hypergraph $H^*$ that corresponds to $G$.
	
	We say that a class $\mathcal{F}$ of objects is \emph{hereditary} if it is downwards closed, meaning that $(A \in \mathcal{F}) \wedge (A' \subset A) \Rightarrow (A' \in \mathcal{F})$. For example, the class of all families of pseudo-discs in the plane is clearly hereditary. 
	
	For two fixed hereditary classes of objects $\mathcal{F},\mathcal{F}'$, we denote by $f(m,k)=f_{\mathcal{F},\mathcal{F}'}(m,k)$ and $f^*(n,\ell)=f^*_{\mathcal{F},\mathcal{F}'}(n,\ell)$ the maximum of $f_G(m,k)$ and of $f^*_G(n,\ell)$ (respectively) over all bipartite graphs $G=G_{A,B}$ such that $A \in \mathcal{F}$, $B \in \mathcal{F}'$, $|A|=m$, and $|B|=n$.\footnote{We note that an extra condition that $G_{A,B}$ is $K_{t,t}$-free could be added here, since all bipartite graphs encountered during the recursive process are $K_{t,t}$-free. It will be interesting to understand whether this additional assumption implies the existence of $\epsilon$-$t$-nets of a smaller size.} Furthermore, we denote by $g(m,n)=g_{\mathcal{F},\mathcal{F}'}(m,n)$ the maximum number of edges in a $K_{t,t}$-free bipartite graph $G_{A,B}$, where $A \in \mathcal{F}$, $B \in \mathcal{F}'$, $|A|=m$, and $|B|=n$. 
	\begin{theorem}\label{thm:recursive}
		Let $\mathcal{F},\mathcal{F}'$ be hereditary classes of objects. In the above notations, we have
		\begin{equation}
			g(m,n) \leq \min_{1 \leq k \leq m-1} \min_{1 \leq \ell \leq n-1} ((k-1)n+(\ell-1) m + g((t-1)f(m,k),(t-1)f^*(n,\ell))). 
		\end{equation} 
	\end{theorem}
	
	\begin{proof}
		For any bipartite graph $G_{A,B}$, where $A \in \mathcal{F}$, $B \in \mathcal{F}'$, $|A|=m$, and $|B|=n$, and any $k,\ell$, we may apply Algorithm~\ref{alg} with $\epsilon=\frac{k}{m}$ and $\epsilon'=\frac{\ell}{n}$, to obtain the bound 
		\[
		|E(G_{A,B})| \leq (k-1)n+(\ell-1)m + |E(G_{A',B'})|.
		\]
		Here, the term $(k-1)n$ bounds the contribution of the `light' vertices in $B$, as there are at most $n$ such vertices and each of them has degree strictly less than $\frac{k}{m}m=k$. The term $(\ell-1)m$ bounds the contribution of the `light' vertices in $A$ in a similar way.
		
		Note that we have $|B'| \leq (t-1)f(m,k)$. Indeed, let $S$ be an $\epsilon$-$t$-net for $H_G$ of size $f(m,k)$. On the one hand, for each $b \in B'$, the hyperedge $e_b$ contains a $t$-tuple from $S$ (since $|e_b|\geq k$ and $S$ is a $\frac{k}{m}$-$t$-net). On the other hand, as $G_{A,B}$ is $K_{t,t}$-free, any $t$-tuple in $S$ participates in at most $t-1$ hyperedges of $H_G$. Thus, $|B'| \leq (t-1)|S|=(t-1)f(m,k)$. By the same argument, we have 
		$|A'| \leq (t-1)f^*(n,\ell)$. Since $\mathcal{F},\mathcal{F}'$ are hereditary, it follows that 
		$
		|E(G_{A',B'})| \leq g((t-1)f(m,k),(t-1)f^*(n,\ell)),
		$
		and thus, 
		\[
		|E(G_{A,B})| \leq (k-1)n+(\ell-1)m + g((t-1)f(m,k),(t-1)f^*(n,\ell)).
		\]
		
		Taking the minimum over all $1 \leq k \leq m-1$ and $1 \leq \ell \leq n-1$ and then the maximum over all $G_{A,B}$, where $A \in \mathcal{F}$, $B \in \mathcal{F}'$, $|A|=m$, and $|B|=n$, completes the proof.
	\end{proof}
	Theorem~\ref{thm:recursive} allows applying results on $\epsilon$-$t$-nets to obtain upper bounds for Zarankiewicz's problem in a black box manner. The results presented in the following subsections are obtained by applying this approach (or parts of it) for specific classes of graphs.
	
	\subsection{Graphs with bounded VC-dimension}\label{subsec:relation}
	
	%In this Section we establish the relation between $\epsilon$-$t$-nets and  %$K_{t,t}$-free bipartite fraphs, and use this relation to 
	%provide a simple short proof of Theorem~\ref{thm:VCour}.
	
	The first step of the approach presented above along with Theorem~\ref{thm:eps-t-net} yield a strikingly simple proof of Theorem~\ref{thm:VCour}.
	
	\begin{proof}[Proof of Theorem~\ref{thm:VCour}]
		Put $\epsilon = \frac{C_1^{1/{d^*}}(t-1)}{m^{1/{d^*}}}$, where $C_1$ is the constant from Theorem \ref{thm:eps-t-net}. Let $N$ be an $\epsilon$-$t$-net for $H_G$ of size $O((d(1+\log t)/\epsilon)\log(1/\epsilon))$, whose existence follows from Theorem~\ref{thm:eps-t-net}.
		
		Let $B'\subset B$ be the set of vertices with degree at least $\epsilon m = \Theta_{d^*,t} (m^{1-\frac{1}{d^*}})$ in $G$. We claim that 
		\[
		|B'| \leq (t-1)|N| =O_{d,t}( \frac{1}{\epsilon}\log \frac{1}{\epsilon})=O_{d,d^*,t}(m^{\frac{1}{d^*} }\log m).
		\]
		%We will show that the number of vertices in $B'$ is very small, i.e.,   Indeed, %assume to the contrary that $| B'| \geq (t-1)|N|+1$. Note that for each vertex $b %\in  B'$, $N[b]$ is a hyperedge in $H_G$  which contains more than $\epsilon m$ %vertices so by Theorem~\ref{thm:eps-t-net} it must contain one of the $t$-tuples %of $N$. However by the pigeon-hole pinciple there must be a $t$-tuple  $T \in N$ %that participates in $t$ or more hyperedges of $H_G$, but this implies the %existence of $K_{t,t}$ in $G_{A,B}$, a contradiction. 
		Indeed, on the one hand, for each $b \in B'$, the hyperedge $e_b$ contains a $t$-tuple from $N$. On the other hand, as $G_{A,B}$ is $K_{t,t}$-free, any $t$-tuple in $N$ participates in at most $t-1$ hyperedges of $H_G$. Thus, $|B'| \leq (t-1)|N|$, as asserted.
		
		To complete the proof, we note that $|E(G)| = (\sum_{b\in B} d(b))$, where $d(b)$ is the degree of $b$ in $G$. Hence, we have
		$$
		|E(G)|=\sum_{b \in B} d(b) = \sum_{b \in B'}d(b) + \sum_{b \in B \setminus B'} d(b) \leq |B'|m+ |B \setminus B'|\epsilon m = O_{d^*,d,t}(m^{1+1/d^*}\log m + nm^{1-1/d^*}). 
		$$ The $min$ 
		%part of the 
		assertion is achieved by applying the same argument to $H_G^*$ instead of $H_G$.
	\end{proof}
	
	A discussion on the relation between our Theorem~\ref{thm:VCour} and Theorem~\ref{thm:main-boundedvc} of Fox et al.~\cite{FPSSZ17} is presented in Appendix~\ref{App:Fox}.
	
	\subsection{$K_{t,t}$-free bipartite intersection graphs of pseudo-discs}
	\label{subsec:pseudo-discs}
	
	%The power of Algorithm \ref{alg} can be optimally exploited where the existence of the $\epsilon$-$t$-net holds already when the number of vertices is `relatively small' with respect to $\frac{1}{\epsilon}$ (e.g., when the number of vertices is about $\Theta(\frac{1}{\epsilon})$). This is the case, for example, in the intersection graph of two families of pseudo-discs. 
	
	Our recursive strategy can be exploited optimally when both the primal and the dual hypergraphs that correspond to $G$ admit $O_t(1/\epsilon)$-sized $\epsilon$-$t$-nets for $\epsilon$ as small as $O(1/|V|)$, where $|V|$ is the number of vertices of the corresponding hypergraph. In this subsection we prove that this is the case for intersection graphs of two families of pseudo-discs.\footnote{The special case $t=2$ of this result appeared in~\cite{AJKSY22}, but the proof method there is very specific for $t=2$.} Then, we use this to obtain an improved linear-sized bound on the number of edges of the graph.
	
	%First, we prove the required $\epsilon$-$t$-net result for the corresponding hypergraph\footnote{The special case $t=2$ of this result appeared in~\cite{AJKSY22}, but the proof method there is specific for $t=2$, and thus, the current proof uses a different argument.}, and then we apply Algorithm~\ref{alg} and analyze the obtained bound. 
	We begin with a formal definition.
	\begin{definition}
		A family $\F$ of simple closed Jordan regions in $\Re^2$ is called a \emph{family of pseudo-discs} if for any $a,b \in \F$, the boundaries of $a$ and $b$ intersect at most twice.
	\end{definition} 
	We prove the following $\epsilon$-$t$-net theorem for a hypergraph induced by two families of pseudo-discs, which might be of independent interest. The theorem is `optimal', in the sense that it provides an $O_t(1/\epsilon)$-sized $\epsilon$-$t$-net already when $\epsilon n$ is constant.
	
	\begin{theorem} \label{thm:pseudo_discs}
		Let $\F_1$ and $\F_2$ be two families of pseudo-discs. Let $H$ be a hypergraph whose vertex-set is $\F_1$, where each $b \in \F_2$ defines a hyperedge $e_b=\{ a \in \F_1 : a \cap b \neq \emptyset  \}$. If $|\F_1|=n$ and $\epsilon n \geq 2t$ then $H$ admits an $\epsilon$-$t$-net of size $O(t^5 \cdot \frac{1}{\epsilon})$.
	\end{theorem}
	
	The proof of Theorem \ref{thm:pseudo_discs} makes use of the following definition and result:
	
	\begin{definition}
		Let $H=(V,\E)$ be a hypergraph. The \emph{Delaunay graph} of $H$ is the graph on the same vertex-set, whose edges are the hyperedges of $H$ of cardinality 2.
	\end{definition}
	
	\begin{theorem}[\cite{AKP21}, Theorem 6(ii,iii)] \label{thm:AKP}
		Let $H=(V,\E)$ be a hypergraph. Suppose there exists $C>0$ such that for every $V' \subset V$, the Delaunay graph of the hypergraph induced by $V'$ has less than $C|V'|$ hyperedges.\footnote{The hypergraph induced by $V'$ is $(V',\E')$, where $\E'=\{e \cap V':e \in \E\}$.} Denote the VC-dimension of $H$ by $d$. Then:
		\begin{enumerate}
			\item $d \leq 2C$.
			
			\item $H$ has $O(t^{d-1}|V|)$ hyperedges of size at most $t$.  
		\end{enumerate}  
	\end{theorem}
	
	\begin{proof}[Proof of Theorem \ref{thm:pseudo_discs}]
		First, we find a `small' set $S \subset V(H)=\F_1$ such that each hyperedge in $H$ of size at least $\epsilon n$ contains at least $t$ elements of $S$. To this end, we first let $K_1 \subset V(H)$ be an $\epsilon$-net for $H$ of size $O(\frac{1}{\epsilon})$. The existence of such a linear-sized $\epsilon$-net is well-known (even for the more general case of two families of non-piercing regions), see, e.g.,~\cite[Thm.~6.2]{AJKSY22}. 
		Then, for each $2 \leq i \leq t$ sequentially, we consider the hypergraph induced on $V(H) \setminus(K_1 \cup \ldots \cup K_{i-1})$, and let $K_i$ be an $O(\frac{1}{\epsilon})$-sized $\frac{\epsilon}{2}$-net for it. Let $S=\bigcup_{i=1}^t K_i$. Clearly, $|S|=O(\frac{t}{\epsilon})$.
		
		We claim that each hyperedge $e \in \E(H)$ with $|e|\geq \epsilon n$ contains at least $t$ elements of $S$. Indeed, if $|S \cap e| <t$, then there exists some $1 \leq i  \leq t$ such that before the $i'$th step, the hyperedge induced by $e$ contained less than $t$ elements of $K_1 \cup \ldots \cup K_{i-1}  $, and $e \cap K_i = \emptyset$. Since $\epsilon n \geq 2t$, we have $\epsilon n -t \geq \frac{\epsilon n}{2}$, a contradiction to $K_i$ being an $\frac{\epsilon}{2}$-net.
		
		Having the set $S \subset V(H)$ in hands, we construct an $\epsilon$-$t$-net for $H$ that consists of $t$-tuples in $\binom{S}{t}$. Let $H_S$ be the hypergraph induced by $H$ on $S$. It is known that the Delaunay graph of $H$, and therefore, of $H_S$, is planar -- namely, the condition of Theorem \ref{thm:AKP} holds with $C=3$, and the VC-dimension of $H_S$ is at most $4$ (see \cite{Kes17+}). These two properties hold for any induced subhypergraph of $H_S$ as well. By Theorem \ref{thm:AKP}, this implies that any induced subhypergraph of $H_S$ on $m$ vertices contains $O(t^3 m)$ hyperedges of size at most $ t$. By a simple double-counting argument, it follows that in any such induced subhypergraph, there exists a vertex that participates in $O(t^4)$ hyperedges of size at most $t$.
		
		Now we are ready to construct the desired $\epsilon$-$t$-net, $N$. We choose a vertex $a$ that participates in $O(t^4)$ hyperedges of size at most $t$ of $H_S$, and add to $N$ all the $t$-sized hyperedges (if exist) that contain $a$. Then we delete $a$, and continue inductively, in the same manner, with the hypergraph induced on $V(H_S) \setminus \{a\}$. We continue in this fashion until all vertices of $S$ are removed. The number of steps is $|S|=O(\frac{t}{\epsilon})$, and each step contributes $O(t^4)$ $t$-tuples to $N$. Hence, in total, we have $|N|=O(t^5 \cdot \frac{1}{\epsilon})$.
		
		We claim that $N$ is an $\epsilon$-$t$-net for $H$. Indeed, consider a hyperedge $e \in \E(H)$ with $|e| \geq \epsilon n$. By the construction of $S$, $e$ contains at least $t$ vertices from $S$. During the process in which we removed one-by-one the vertices of $S$, consider the step in which the size $|e \cap S|$ was reduced from $t$ to $t-1$. At this step, $e$ contained exactly $t$ elements from $S$, that formed a $t$-tuple added to $N$. Hence, $e$ contains a $t$-tuple from $N$. Thus, $N$ is an $\epsilon$-$t$-net for $H$. This completes the proof of Theorem \ref{thm:pseudo_discs}.
	\end{proof}
	
	\begin{remark}\label{rem:hl}
		It is clear from the proof of Theorem~\ref{thm:pseudo_discs} that the theorem holds (up to a factor of $O_t(1)$) for any hypergraph $H$ that satisfies the following properties:
		\begin{enumerate}
			\item Any induced subhypergraph $H' \subset H$ admits an $\epsilon$-net of size $O(1/\epsilon)$, for any $\epsilon \geq \frac{t}{|V(H')|}$.
			
			\item $H$ has a hereditarily linear Delaunay graph.
		\end{enumerate} 
		An example of such a setting is the intersection hypergraph of two families of non-piercing regions in the plane, discussed in Section~\ref{rem:non-piercing-regions}.
	\end{remark}
	
	The existence of a `small'-sized $\epsilon$-$t$-net for the intersection hypergraph of pseudo-discs enables us to apply Algorithm \ref{alg}, and thus to obtain an improvement and a generalization of the result of Chan and Har-Peled~\cite{CH23} mentioned in the introduction.
	The following lemma quantifies the partition of the vertices in steps 4-5 of Alg.~\ref{alg} into `heavy' and `light' ones.
	\begin{lemma}\label{lem:pd}
		Let $\F_1,\F_2$ be two families of pseudo-discs with $|\F_1|=|\F_2|=n$, and let $G=G_{\F_1,\F_2}$ be the bipartite intersection graph of $\F_1,\F_2$. If $G$ is $K_{t,t}$-free and $\ell \geq 2t$, then the number of vertices in $\F_1 \cup \F_2$ whose degree in $G$ is at least $\ell$ is $O(t^6 \frac{n}{\ell})$.
	\end{lemma}
	
	\begin{proof}[Proof of Lemma \ref{lem:pd}]
		We prove the lemma w.l.o.g.~for the vertices in $\F_2$. Let $\epsilon=\frac{\ell}{n}$. Since $\epsilon n = \ell \geq 2t$, we can apply Theorem \ref{thm:pseudo_discs} to obtain an $\epsilon$-$t$-net $N$ of size $O(\frac{t^5}{\epsilon})$ for the primal hypergraph $H_G$. Each hyperedge of $H_G$ of size at least $\epsilon n =\ell$ contains a $t$-tuple from $N$, but since $G$ is $K_{t,t}$-free, each such a $t$-tuple participates in at most $t-1$ hyperedges.
		
		Therefore, the total number of hyperedges of size at least $\ell$ in $H_G=(\F_1,\E_{\F_2})$ is at most  $(t-1)O(\frac{t^5}{\epsilon})=O(\frac{t^6}{\epsilon})=O(\frac{t^6 n}{\ell})$. This is exactly the number of vertices in $\F_2$ with degree at least $\ell$ in $G$. This completes the proof of Lemma \ref{lem:pd}.
	\end{proof}
	
	Now we are ready to prove Theorem \ref{thm:pd-intro}. The idea of the proof is to apply Algorithm~\ref{alg} with such a choice of $\epsilon,\epsilon'$ in step 1 (which is actually done by choosing the parameter $\ell$ in Lemma \ref{lem:pd}), that the number of `heavy' vertices\footnote{The notions of `heavy' and `light' vertices were explained in the correctness proof of Algorithm~\ref{alg}.} in each part of the graph is reduced by a factor of 2. 
	Intuitively, this factor-2-reduction saves the order of magnitude of the total sum from being affected by the number of steps in the recursive process. This choice can be made since (unlike in the general case of Theorem \ref{thm:eps-t-net}), Theorem \ref{thm:pseudo_discs} holds already 
	when $\epsilon n \geq 2t$. 
	
	\begin{proof}[Proof of Theorem \ref{thm:pd-intro}]
		Denote by $f(n)$ the maximum possible number of edges in $G_{\F_1,\F_2}$, for $\F_1,\F_2$ as in the statement of the theorem. Let $C \geq 1$ be a universal constant such that Lemma \ref{lem:pd} holds with $C \frac{t^6 n}{\ell}$. We prove by induction that our claim holds with $f(n) \leq 8 C t^6 n$.
		
		For $n \leq 8 C t^6$, the assertion is trivial since $|E(G_{\F_1,\F_2})| \leq n^2$. We assume correctness for $\frac{n}{2}$ and prove the assertion for $n$. By Lemma \ref{lem:pd} with $\ell=2Ct^6 \geq 2t$, the number of vertices in $G_{\F_1,\F_2}$ with degree at least $2Ct^6$ is at most $C \frac{t^6 n}{\ell}=\frac{n}{2}$. 
		
		Recall that a vertex in $\F_1 \cup \F_2$ is called `heavy' if its degree in $G_{\F_1,\F_2}$ is at least $\ell$, and otherwise, it is called `light'.
		There are at most $n \ell$ edges in $G_{\F_1,\F_2}$ that connect a light vertex of $\F_1$ (resp., $\F_2$) with some vertex of $\F_2$ (resp., $\F_1$). The number of edges in $G_{\F_1,\F_2}$ that connect two heavy vertices is at most $f(\frac{n}{2})$, and by the induction hypothesis, $f(\frac{n}{2}) \leq 4Ct^6n$. Therefore, the total number of edges in $G_{\F_1,\F_2}$ is at most $(2Ct^6+2Ct^6+4Ct^6)n=8Ct^6n$.
	\end{proof}
	
\subsection{$K_{t,t}$-free Intersection Graph of Families of Non-Piercing Regions}
\label{rem:non-piercing-regions}

Theorem \ref{thm:pd-intro} can be readily generalized (albeit, with a slightly weaker bound of $O(t^8n)$) to the more general setting of bipartite intersection graphs of two families of \emph{non-piercing regions in the plane} -- i.e., families $\F$ of regions such that for any $S,T \in \F$, $S \setminus T$ is connected. 

To see this, we may repeat the proof of the theorem step-by-step. First, as follows from Remark~\ref{rem:hl} (except for the exact dependence on $t$ which we compute here), Theorem~\ref{thm:pseudo_discs} can be generalized to this setting with the weaker bound of $O(t^7 \cdot \frac{1}{\epsilon})$.
Indeed, the existence of an $O(\frac{1}{\epsilon})$-sized $\epsilon$-net in this setting is well-known, as was mentioned at the beginning of the proof of Theorem~\ref{thm:pseudo_discs}. 	
Raman and Ray (who introduced the notion of non-piercing regions in~\cite{RR20}) showed that the Delaunay graph of the intersection hypergraph of two families of non-piercing regions is planar, and thus, Theorem~\ref{thm:AKP} implies that the number of hyperedges of size at most $t$ in any induced hypergraph on $m$ vertices is at most $O(t^5m)$. (Note that this bound is weaker than the bound $O(t^3m)$ which we obtain in the proof of Theorem~\ref{thm:pseudo_discs}, since here we do not know that the VC-dimension of the hypergraph is at most $4$). The rest of the proof of Theorem~\ref{thm:pseudo_discs} applies verbatim, and so does the derivation of Theorem~\ref{thm:pd-intro} from Theorem~\ref{thm:pseudo_discs}. The obtained bound is $O(t^8n)$ instead of $O(t^6n)$, since the factor $t^2$ loss carries along the proof.

	\section{$K_{t,t}$-free Bipartite Intersection Graphs of Axis-parallel Rectangles}
	\label{sec:rectangles}
	
	In this section we prove Theorem~\ref{thm:rect-intro} -- a sharp upper bound on the number of edges in a $K_{t,t}$-free bipartite intersection graph of two families of axis-parallel rectangles. As was explained in the introduction, the current knowledge on $\epsilon$-$t$-nets for intersection hypergraphs of families of axis-parallel rectangles is not sufficient for obtaining a sharp bound for Zarankiewicz's problem using our $\epsilon$-$t$-net approach. Hence, we prove the theorem by an entirely different method that uses the sharp bound of Chan and Har-Peled on the number of edges in a $K_{t,t}$-free incidence graph of points and axis-parallel rectangles~\cite{CH23}, along with other combinatorial and geometric techniques. Then, we use similar techniques to obtain an improved bound for $K_{t,t}$-free incidence graphs of points and pseudo-discs.
	
	\subsection{Proof of Theorem~\ref{thm:rect-intro}} 
	
	Let us restate the theorem, in a slightly stronger form.   
	
	%without using the relation established above to $\eps$-$t$-nets. Let us recall a bit more general statement of it:
	
	\medskip 
	\noindent \textbf{Theorem \ref{thm:rect-intro}.}
	Let $t \geq 2$.
	Let $A,B$ be two families of axis-parallel rectangles, $|A|=n,|B|=m$, s.t. $A \cup B$ is in general position. If $G_{A,B}$ is $K_{t,t}$-free, then $$|E(G_{A,B})|=O\left(t(n+m) \frac{\log (n+m)}{\log \log (n+m)} + t^6 n+tm) \right).$$

	\begin{proof}
		Any intersection between $a \in A$ and $b \in B$ belongs to exactly one of four types:
		\begin{enumerate}
			\item The rectangle $a$ is strictly contained in the rectangle $b$.
			
			\item The rectangle $b$ is strictly contained in the rectangle $a$.
			
			\item A vertical edge of $b$ intersects a horizontal edge of $a$.
			
			\item A vertical edge of $a$ intersects a horizontal edge of $b$.
		\end{enumerate}
		%The two other types 1',2' are symmetric by replacing $a$ and $b$.
		
		We bound separately the numbers of intersections of each of these types. 
		
		\medskip \noindent \textbf{Intersections of type~1.} We define a bipartite graph $G$ whose vertices are all the corners of rectangles in $A$, and all the rectangles in $B$. A corner $x$ is adjacent to a rectangle $b \in B$ if $x \in b$. Clearly, $|V(G)|=4n+m$. We observe that $G$ is $K_{4t-3,4t-3}$-free, since if some $4t-3$ corners are all contained in the same $4t-3$ rectangles of $B$, then these $4t-3$ corners belong to at least $t$ different rectangles in $A$, and this contradicts the assumption that $G_{A,B}$ is $K_{t,t}$-free. Therefore, we can apply the following result of Chan and Har-Peled~\cite{CH23}:
		\begin{lemma}[\cite{CH23}, Lemma 4.4]\label{lem:CH}
			Let $P$ be a set of $n$ points in $\Re^2$, and let $\R$ be a family of $m$ axis parallel rectangles in $\Re^2$. If the incidence graph $G_{P,\R}$ is $K_{t,t}$-free, then $E(G_{P,\R})=O_{\epsilon}(tn \frac{\log n}{\log \log n}+ tm \log^{\epsilon}n)$, for any constant $\epsilon>0$.
		\end{lemma}

		\noindent Since the two sides of $V(G)$ contain $4n$ and $m$ vertices, respectively, Lemma~\ref{lem:CH} (applied with any fixed $0<\epsilon<1$) yields $$|E(G)|=O\left(t(n+m) \frac{\log (n+m)}{\log \log (n+m)}\right).$$ 
		
		Note that each intersection of type 1 contributes exactly 4 edges to $G$. Hence, the number of intersections of type 1 is $O(t(n+m) \frac{\log (n+m)}{\log \log (n+m)})$. By symmetry, the same bound holds for the number of intersections of type 2.
		
		\medskip \noindent \textbf{Intersections of type 3.}
		Define a bipartite graph $K=K_{S,S'}$ whose vertices are the horizontal edges of rectangles in $A$ (that we call \emph{horizontal vertices of $K$}), and the vertical edges of rectangles in $B$ (that we call \emph{vertical vertices of $K$}). A vertical vertex is adjacent to a horizontal vertex if the corresponding edges cross. Each intersection of type 3 contributes either 1, 2, or 4 edges to $H$. Therefore, the number of such intersections is at most $4|E(K)|$.
		
		Denote the vertices of $S$ (i.e., the vertical vertices of $K$) by $v_1,v_2,\ldots,v_m$ (in an arbitrary order). For $1 \leq i \leq 2m$, let $d_i$ be the degree of $v_i$ in $K$. Clearly, $|E(K)|= \Sigma_{i=1}^{2m}d_i$. Let $\F \subset \binom{S'}{2t-1}$ be the family of all \emph{canonical $(2t-1)$-tuples} of horizontal vertices, where a $(2t-1)$-tuple $T$ of horizontal vertices is called canonical if there exists some vertical segment $L$ (not necessarily a vertical vertex!) that intersects exactly the vertices of $T$ among all the horizontal vertices (i.e., we have $\{x \in S': x \cap L \neq \emptyset\}=T$).
		
		\begin{claim}\label{cl:sizeF}
			In the above notations, $|\F|=O(t^5n)$.
		\end{claim} 
		
		We leave the proof of Claim \ref{cl:sizeF} to the end of this section, and continue with the proof of Theorem \ref{thm:rect-intro} (assuming the claim).
		
		For each $1 \leq i \leq 2m$, we define $x_i$ to be the number of canonical $(2t-1)$-tuples of horizontal vertices which the vertical vertex $v_i$ intersects. That is,
		$$ x_i=|\{\{h_1,\ldots,h_{2t-1}\} \in \F :  \forall 1\leq j \leq 2t-1, v_i \cap h_j \neq \emptyset\}|.$$ 
		%where each $h_j$ is a horizontal vertex of $K$. 
		We would like to obtain lower and upper bounds on $\Sigma_{i=1}^{2m}x_i$. 
		
		On the one hand, $\Sigma_{i=1}^{2m}x_i \leq (2t-2) |\F|$. Indeed, for any canonical $(2t-1)$-tuple $\{h_1,\ldots,h_{2t-1}\} \in \mathcal{F}$, at most $2t-2$ $v_i$'s intersect all of $h_1,\ldots,h_{2t-1}$, since otherwise,  $G_{A,B}$ contains $K_{t,t}$ as a subgraph (as the $h_j$'s must belong to at least $t$ different rectangles in $A$ and the $v_i$'s must belong to at least $t$ different rectangles in $B$). By Claim \ref{cl:sizeF}, this implies 
		\begin{equation}\label{eq:UB}
			\Sigma_{i=1}^{2m}x_i=(2t-2) \cdot O(t^5 n) = O(t^6n).
		\end{equation}
		On the other hand, for each $1 \leq i \leq 2m$, we have $d_i-2t+2 \leq x_i$. Indeed, if $d_i \leq 2t-2$ the inequality is trivial. If $d_i \geq 2t-1$ and $v_i$ intersects (w.l.o.g.) the horizontal vertices $h_1,\ldots,h_{d_i}$ in this order, then each consecutive $(2t-1)$-subsequence of $h_1,\ldots,h_{d_i}$ belongs to $\F$, since it is the intersection of some subsegment of $v_i$ with the set of horizontal vertices. Therefore, 
		\begin{equation}\label{eq:LB}
			\Sigma_{i=1}^{2m} (d_i-2t+2) \leq \Sigma_{i=1}^{2m}x_i.
		\end{equation}
		Combining (\ref{eq:UB}) and (\ref{eq:LB}) together, we obtain 
		$$\Sigma_{i=1}^{2m} (d_i-2t+2) = O(t^6 n ),$$
		and thus,
		$$|E(K)|=\Sigma_{i=1}^{2m} d_i = O(t^6n +tm).$$
		Hence, the number of intersections of type 3 is $O(t^6n+tm)$. By symmetry, the same bound applies to the number of intersections of type 4. This completes the proof of the theorem (assuming Claim~\ref{cl:sizeF}).
	\end{proof}
	
	The only part that remains is the proof of Claim \ref{cl:sizeF}.
	%To this end we need a more detailed version of Theorem \ref{thm:AKP}, also proved in \cite{AKP21}:
	
	%\medskip
	
	%\noindent \textbf{Theorem \ref{thm:AKP} - extended version (\cite{AKP21}).}
	%	Let $H=(V,\E)$ be a hypergraph with VC-dimension $d$. Suppose there exists $C>0$ such that for every $V' \subset V$, the Delaunay graph of the hypergraph induced by $V'$ has $<C|V'|$ edges. Then $d \leq 2C$ and $H$ has $O(t^{d-1}|V|)$ hyperedges of size$\leq t$.   
	
	\begin{proof}[Proof of Claim \ref{cl:sizeF}.]
		Define a hypergraph $J$ whose vertices are the horizontal edges of the rectangles in $A$ --- $h_1, \ldots, h_{2n}$, and each vertical segment $v$ (which is not necessarily a vertical vertex!) defines a hyperedge $e_v$ which is the subset of  $\{h_1, \ldots, h_{2n}\}$ that $v$ intersects. 
		
		Note that $\F$ is the set of hyperedges of size $2t-1$ of $J$. We would like to bound the size of this set by $O(t^5n)$. To this end,  we prove that the Delaunay graph of $J$, $Del(J)$, has a hereditarily linear number of edges. It is clearly sufficient to prove that $|E(Del(J))|$ is linear in $|V(Del(J))|=2n$. 
		
		Let us describe a planar drawing of $Del(J)$. We represent each vertex of $Del(J)$ by the right endpoint of the corresponding horizontal edge. Each edge $v=\{h_i,h_j\} \in E(Del(J))$ is drawn as a 3-polygonal path that starts at the right endpoint of $h_i$, continues on the subsegment of $v$ that connects $h_i$ and $h_j$, and continues on $h_j$ towards its right endpoint (see Figure \ref{fig:fig1}).
		
		\begin{figure}[tb]		
			\begin{center}
				\scalebox{0.7}{
					\includegraphics{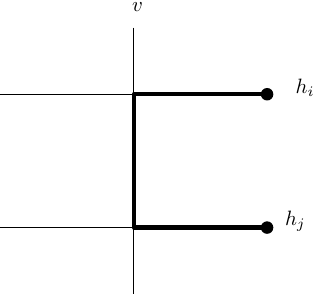}
				}
			\end{center}
			\caption{The planar drawing of $v=\{h_i,h_j\}$ (in bold).}
			\label{fig:fig1}
		\end{figure}
		
		The drawing described here is a planar drawing of $Del(J)$. It is easy to verify that if for some four distinct vertices $h_i,h_j,h'_i,h'_j$, the planar drawing of $v=\{h_i,h_j\}$ intersects the planar drawing of $v'=\{h_i',h_j'\}$, then either $v$ or $v'$ cannot be an edge of $Del(J)$ (see Figure \ref{fig:fig2}). As $Del(J)$ admits a planar drawing in which no two vertex-disjoint edges intersect, it is planar (e.g., by an easy special case of the Hanani-Tutte theorem).

		\begin{figure}[tb]		
			\begin{center}
				\scalebox{0.7}{
					\includegraphics{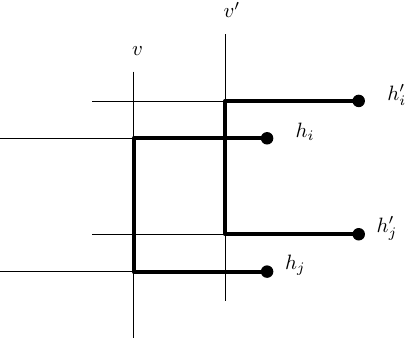}
				}
			\end{center}
			\caption{In this figure, the planar drawing of $\{h_i,h_j\}$ intersects the planar drawing of $\{h_i',h_j'\}$. However, $v'$ intersects three horizontal segments, and therefore $\{h_i',h_j'\}  \notin E(Del(J))$.}
			\label{fig:fig2}
		\end{figure}
		
		Thus, each subgraph of $Del(J)$ on $\ell$ vertices contains at most $3 \ell -6$ edges. By Theorem \ref{thm:AKP}, this implies that the number of hyperedges of size at most $2t-1$ in $J$ is $O(t^5n)$. In particular, $|\F|=O(t^5n)$, as asserted.
	\end{proof}
	
	\begin{remark}
		Note that some intersections between two rectangles $a \in A, b \in B$ were counted twice. In the counting of type 1, we actually counted all the intersections in which a corner of $a$ is contained in $b$, and not only the ones in which $a \subseteq b$. 
	\end{remark}
	
	\begin{remark}
		The proof of Theorem~\ref{thm:rect-intro} readily implies that the number of edges in a $K_{t,t}$-free bipartite intersection graph $G_{A,B}$ of two families of axis-parallel frames (i.e., boundaries of rectangles) with $|A|=n, |B|=m$ is $O(t^6n+tm)$. Indeed, the only possible types of intersections between a pair of axis-parallel frames are intersections of types 3, 4 presented above. In the proof, the number of such intersections is bounded by $O(t^6n+tm)$.  
	\end{remark}

	\subsection{$K_{t,t}$-free incidence graph of points and pseudo-discs}\label{sec:ptpd}
	
	As mentioned already, our Theorem \ref{thm:pd-intro} both improves and generalizes to intersection graphs the following result of \cite{CH23}:
	
	\begin{theorem}[\cite{CH23}, Corollary 5.8] \label{thm:CHpd}
		Let $P$ be a set of $n$ points in the plane, and let $\F$ be a set of $m$ $y$-monotone pseudo-discs in the plane. If the incidence graph $G_{P, \F}$ is $K_{t,t}$-free, then $|E(G)|=O(
		tn + tm(\log \log m + \log t))$.
	\end{theorem}
	
	In the setting of Theorem \ref{thm:CHpd}, where one of the families consists of points in $\Re^2$, we can obtain an improved bound (in terms of the dependence on $t$), using a variant of the method presented above.
	%improve the dependance in $t$ from Theorem \ref{thm:pd-intro}, even without using the relation to $\eps$-$t$-nets. 
	We prove:
	\begin{theorem}\label{thm:pts_pd}
		Let $A$ be a family of $n$ points in the plane, and let $B$ be a family of $m$ pseudo-discs. If $G_{A,B}$ is $K_{t,t}$-free, then $|E(G_{A,B})|= O (t^4n+tm)$.
	\end{theorem}

	\begin{proof}
		We use the following lemma proved by Pinchasi~\cite{Pin14}.
		\begin{lemma}[\cite{Pin14}]\label{Lemma:Rom1}
			Let $\D$ be a family of pseudo-discs. Let $d \in \D$ and let $x \in d$ be any point. Then $d$ can be continuously shrunk to the point $x$ in such a way that at each moment during the shrinking process, $\D$ continues to be a family of pseudo-discs.
		\end{lemma}
		
		We enlarge the family $B$ by shrinking each pseudo-disc $b \in B$ toward each $a \in A$ which is contained in $b$, and adding to $B$ representatives of all objects obtained in the process. Specifically, for each $a\in A$, we apply Lemma~\ref{Lemma:Rom1} sequentially, where every time we stop the shrinking process once the shrunk object `loses' the first point, add the shrunk object to the family, and repeat the process, starting with this `shrunk pseudo-disc' and the same point $a$. We say that two objects belong to the same equivalence class if they contain the same set of points. We define $B'$ to be a family that contains one representative from each equivalence class of objects encountered during the process (for all points $a \in A$). 
		Since the equivalence classes correspond to subsets of $A$, then the resulting family, $B'$, is a finite family of pseudo-discs.\footnote{This process, which has already become standard, is decribed in~\cite[Corollary~1]{Pin14}.} We call a $t$-tuple of points in $A$, $\{a_1,\ldots,a_t\}$, \emph{a canonical $t$-tuple} if there exists some $b \in B'$ such that $b \cap A=\{a_1,\ldots,a_t\}$.
		
		\iffalse
		We enlarge the family $B$ by shrinking each pseudo-disc $b \in B$ toward each $a \in A$ which is contained in $b$, and adding to $B$ a representative pseudo-disc from each equivalence class of the shrunk copies. Since the equivalence classes correspond to subsets of $A$, then the resulting family, $B'$, is a finite family of pseudo-discs. We call a $t$-tuple of points in $A$, $\{a_1,\ldots,a_t\}$, \emph{a canonical $t$-tuple} if there exists some $b \in B'$ such that $b \cap A=\{a_1,\ldots,a_t\}$.
		\fi
		
		Denote $B=\{b_1,b_1,\ldots,b_m\}$. For each $1 \leq i \leq m$, let $d_i$ be the number of points from $A$ that $b_i$ contains. Let 
		$$\F = \{  \{a_1,\ldots,a_t\} \subset A:    \exists b \in B' \mbox{ s.t. }b \cap A =  \{a_1,\ldots,a_t\}  \}$$
		be the collection of all the canonical $t$-tuples from $A$. Let $x_i$ ($1 \leq i \leq m$) be the number of $t$-tuples from $\F$ that are contained in $b_i$. 
		
		On the one hand, since each $t$-tuple in $\F$ is contained in at most $t-1$ elements of $B$, we have 
		$$	\Sigma_{i=1}^m x_i \leq (t-1) |\F|.$$
		By \cite{BPR13}, we have $|\F|=O(t^2n)$, and therefore,
		\begin{equation}\label{eq:pdUB}
			\Sigma_{i=1}^m x_i =O(t^3n).
		\end{equation}
		
		On the other hand, if some $b_i \in B$ contains at least $t$ points of $A$, then any $a \in A$ that is contained in $b_i$ belongs to at least one canonical $t$-tuple that is included in $B$ (since $b_i$ can be continuously shrunk towards $a$, so $B'$ contains some shrunk copy of it that contains exactly $t$ points of $A$). Therefore, $\forall 1 \leq i \leq m , \lfloor \frac{d_i}{t} \rfloor \leq x_i$. (Note that this inequality holds even when $d_i < t$). Hence,
		\begin{equation}\label{eq:pdLB}
			\Sigma_{i=1}^m (\frac{d_i}{t}-1)    \leq \left\lfloor \frac{d_i}{t} \right\rfloor  \leq \Sigma_{i=1}^m x_i .
		\end{equation}
		Combining (\ref{eq:pdUB}) and (\ref{eq:pdLB}) together and rearranging, we obtain
		$$|E(G_{A,B})| =  \Sigma_{i=1}^m d_i =O(t^4n+tm) .$$
		This completes the proof.
	\end{proof}

	\begin{remark}
		Theorem \ref{thm:pts_pd} holds if we replace $A$ with a collection of pairwise disjoint sets, as was done in \cite{KelS20}. 
	\end{remark}

\section*{Acknowledgments}

The authors are grateful to Aleksa Milojevi\'{c} for communicating to them the results of~\cite{HMST24}. The research presented in this paper was partially supported by the Israel Science Foundation (grant no.~1065/20) and by the United States -- Israel Binational Science Foundation (NSF-BSF grant no.~2022792).

	\bibliographystyle{plain}
	\bibliography{references-klan}

	\appendix
	
	\section{Comparison between Theorem~\ref{thm:main-boundedvc} and Theorem~\ref{thm:VCour}}
	\label{App:Fox}
	
	In this appendix we make several remarks on the relation between our Theorem~\ref{thm:VCour} and Theorem~\ref{thm:main-boundedvc} of Fox et al.~\cite{FPSSZ17}. Throughout this appendix, we assume w.l.o.g.~that the bound of Theorem~\ref{thm:main-boundedvc} is $O(nm^{1-1/d^*})$. The other case is handled in the same way.
	
	\medskip \noindent \textbf{The bounds are equal up to a constant in main cases of interest.}
	The bound of our Theorem~\ref{thm:VCour} matches the  bound of Theorem~\ref{thm:main-boundedvc} (up to constant factors) whenever $d^*>2$ and $n = \Omega(m^{2/d^*} \log m)$ (and in particular, when $d^*>2$ and $m=n$), as in these cases we have  
	$m^{1+1/d^*} \log m = O(nm^{1-1/d^*})$. 
	
	\medskip \noindent \textbf{Our bound can be enhanced by applying recursion.}
	In cases where our bound is weaker than the bound of Theorem~\ref{thm:main-boundedvc}, the term $O(m^{1+1/d^*}\log m)$ which comes from the contribution of `heavy' vertices can be reduced by applying subsequent steps of the recursive strategy (using the fact that having dual VC-dimension at most $d^*$ is a hereditary property). As these cases seem less interesting, we omit the calculation.   
	
	\medskip \noindent \textbf{The proof strategies are `dual' to each other.}
	Our strategy is somewhat `dual' to the proof strategy of Theorem~\ref{thm:main-boundedvc}. Roughly speaking, the central ingredient in the proof of Theorem~\ref{thm:main-boundedvc} is showing that there exists a vertex $a \in A$ whose degree is at most $O(nm^{-1/d^*})$. Then, the proof proceeds iteratively by removing such a `light' vertex along with all edges incident to it, and finding a new `light' vertex in the remaining graph. The process ends once all vertices of $A$ were removed, and so, all edges of the graph were counted. As the number of vertices in $A$ is reduced by 1 at every step, the total number of edges is at most $O(\sum_{i=1}^m ni^{-1/d^*})= O(nm^{1-1/d^*})$. 
	
	Our proof can also be viewed as an iterative `vertex removal' process, where we first remove the `light' vertices and then remove the `heavy' vertices. However, the vertices we remove belong to $B$ and not to $A$; at the first $n-O(n^{1/d^*})$ steps, vertices with degree less than $O(m^{1-1/d^*})$ are removed, and at the last $O(n^{1/d^*})$ steps, the degree of the removed vertices may be as large as $m$. The proof strategies cannot be combined, as the `vertex removal' processes are performed on different sides of the graph. (Of course, in our strategy we could perform the process on the other side of the graph, but then the bound would become $O(mn^{1-1/d})$ which we assumed above to be inferior to the bound $O(nm^{1-1/d^*})$).
	
	\medskip \noindent \textbf{The proof strategies can be combined when $n=m$ and $d^*=d$.}
	In the special case where $n=m$ and $d^*=d$, we can perform the `vertex removal' process at the same side as in the proof of Fox et al.~(since applying the process on both sides leads to the same bound). In this case, one can combine our proof with the proof of Fox et al.~to obtain a slightly stronger bound, by checking at each step, which strategy provides a lower degree of the `removed' vertex. At the first steps, the strategy of Fox et al.~might give a better bound (this depends on the hidden constants in the two proofs). When the number of remaining steps becomes $o(n)$, our strategy certainly becomes superior (as the degree it provides is still $O(n^{1-1/d})$, while the degree at Fox et al.'s strategy is already at a larger order of magnitude). At the last $O(n^{1/d})$ steps, the strategy of Fox et al.~is superior once again (as the bound we use at these steps is $n$). Thus, a combination of the strategies yields a better bound. A schematic graph of the upper bounds on the degrees of removed vertices in both strategies is given in Figure~\ref{fig:fig5}.
	
%	It should however be noted that in this special case, a stronger bound of $o(n^{2-1/d})$ was obtained by Janzer and Pohoata~\cite{JP20}.
	
	\begin{figure}[tb]		
		\begin{center}
			\scalebox{0.7}{
				\includegraphics{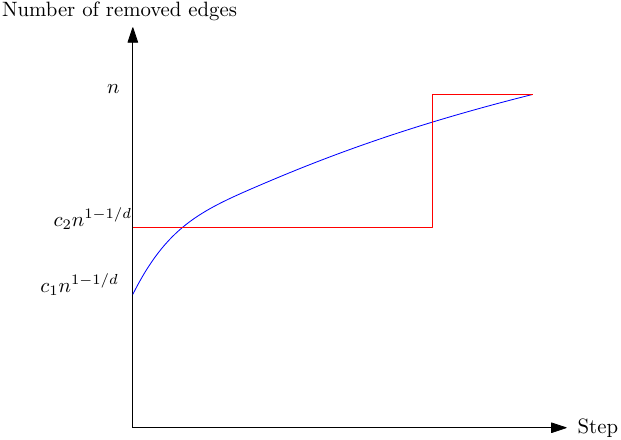}
			}
		\end{center}
		\caption{An illustration for the discussion in Appendix \ref{App:Fox} in the case $m=n, d=d^*$. The blue graph represents the method of \cite{FPSSZ17}, and the red one represents the method suggested in this work.}
		\label{fig:fig5}
	\end{figure}
	
\end{document}